\documentclass[12pt, reqno, fleqn]{amsart}
\usepackage{amsmath, amsthm, amscd, amsfonts, amssymb, graphicx, color}
\usepackage[bookmarksnumbered, colorlinks, plainpages]{hyperref}

\input{mathrsfs.sty}
\newtheorem{theorem}{Theorem}[section]
\newtheorem{lemma}[theorem]{Lemma}
\newtheorem{proposition}[theorem]{Proposition}
\newtheorem{corollary}[theorem]{Corollary}
\theoremstyle{definition}

\theoremstyle{remark}
\newtheorem{remark}[theorem]{Remark}
\numberwithin{equation}{section}

\begin{document}
\title[Berezin number inequalities for  Hilbert space operators]{ Berezin number inequalities for  Hilbert space operators}
\author[M. Bakherad and M.T. Karaev ]{Mojtaba Bakherad$^1$ and Mubariz T. Karaev$^2$}
\address{$^1$Department of Mathematics, Faculty of Mathematics, University
of Sistan and Baluchestan, Zahedan, I.R.Iran.}
\email{mojtaba.bakherad@gmail.com; bakherad@member.ams.org}
\subjclass[2010]{Primary: 15A60, Secondary:  47B20.}
\keywords{Reproducing kernel, numerical range, numerical radius.}
\address{$^2$Department of Mathematics, College of Science, King Saud University, P.O.Box 2455, Riyadh 11451, Saudi Arabia}
\email{mgarayev@ksu.edu.sa}
\begin{abstract}

In this paper, by using of the definition Berezin symbol,  we show some Berezin number inequalities. Among other inequalities, it is shown that if
$A, B, X\in{\mathbb{B}}(\mathscr H)$, then%
$$\mathbf{ber}(AX\pm XA)\leqslant \mathbf{ber}^{\frac{1}{2}}\left(A^*A+AA^*\right)\mathbf{ber}^{\frac{1}{2}}\left(X^*X+XX^*\right)$$
and %
$$\mathbf{ber}^2(A^*XB)\leqslant\|X\|^2\mathbf{ber}(A^*A)\mathbf{ber}(B^*B).$$
\end{abstract}

\maketitle



\section{Introduction}
Let ${\mathbb{B}}(\mathscr H)$ denote the $C^{\ast }$-algebra
of all bounded linear operators on ${\mathscr H}$ with the identity $I$.
A functional Hilbert space is a Hilbert space ${\mathscr H}={\mathscr H}(\Omega)$ of complex-valued
functions on a  set $\Omega$, which has the property that point evaluations are continuous
i.e., for each $\lambda\in\Omega$ the map $f\longmapsto f(\lambda)$ is a continuous linear functional on ${\mathscr H}$.   Berezin set and Berezin number of the operator $A$ are defined by
 $\mathbf{Ber}(A):=\big\{\widetilde{A}(\lambda):\lambda\in\Omega\big\}$ and $\mathbf{ber}(A):=\sup\big\{|\widetilde{A}(\lambda)|:\lambda\in \Omega\big\}$,
respectively. It is clear that
the Berezin symbol $\widetilde{A}$  is the bounded function on $\Omega$ whose values lies in the numerical range of
the operator $A$ and hence $ \mathbf{Ber}(A)\subseteq W(A)$(numerical radius) and $\mathbf{ber}(A)\leqslant w(A)$(numerical range)
for all $A\in{\mathbb{B}}(\mathscr H)$. The Berezin number of an operator $A$  satisfies the following properties:
\begin{align}\label{prop}
&({\rm i})\,\,\mathbf{ber}(A)\leqslant \|A\|.\\&
({\rm ii})\,\,\mathbf{ber}(\alpha A)=|\alpha| \mathbf{ber}(A)\,\,\textrm{for all}\,\, \alpha\in {\mathbb{C}}.\nonumber\\&
({\rm iii})\,\,\mathbf{ber}(A+B)\leqslant  \mathbf{ber}(A)+ \mathbf{ber}(B)\nonumber.
\end{align}
 The Berezin symbol  is widely applied in the various questions of uniquely determines
the operator and analysis. For further information about Berezin symbol we refer the reader to \cite{englis, kar3, kar4, nor} and references therein.\\
In this paper, by using some ideas of \cite{ host, yam},  we present several Berezin number inequalities. In particular, we obtain the inequalities
\begin{align*}
&({\rm i})\,\,\mathbf{ber}(AX\pm XA)\leqslant \mathbf{ber}^{\frac{1}{2}}\left(A^*A+AA^*\right)\mathbf{ber}^{\frac{1}{2}}\left(X^*X+XX^*\right);\\&
({\rm ii})\,\,\mathbf{ber}(A^*XB+B^*YA)\leqslant2\sqrt{\|X\|\|Y\|}\mathbf{ber}^{\frac{1}{2}}\left(B^*B\right)\mathbf{ber}^{\frac{1}{2}}\left(AA^*\right),
\end{align*}
where $A, B,   X, Y\in {\mathbb B}({\mathscr H}(\Omega))$.
\section{The results}
\bigskip To prove our first result, we need the following lemma.
\begin{lemma}\label{L1}
Let $X\in{\mathbb B}({\mathscr H}(\Omega))$. Then
\begin{align*}
\mathbf{ber}(X)=\underset{\theta\in
\mathbb{R}
}{\sup}\,\mathbf{ber}\left(\mathfrak{{Re}}(e^{i\theta}X)\right)=\underset{\theta\in
\mathbb{R}
}{\sup}\,\mathbf{ber}\left({\mathfrak{Im}}(e^{i\theta}X)\right),
\end{align*}
where $\mathfrak{Re}(X)=\frac{X+X^*}{2}$ and $\mathfrak{Im}(X)=\frac{X-X^*}{2i}$.
\end{lemma}
\begin{proof}
Let $\hat{k}_{\lambda}$ be  the normalized reproducing kernel of ${\mathscr H(\Omega)}$.
It follows from
$$\underset{\theta\in
\mathbb{R}
}{\sup}\,\left\langle \mathfrak{Re}(e^{i\theta}X)\hat{k}_{\lambda}, \hat{k}_{\lambda}\right\rangle=\left|\left\langle X\hat{k}_{\lambda}, \hat{k}_{\lambda}\right\rangle\right|$$
 that
 \begin{align*}
\underset{\theta\in
\mathbb{R}
}{\sup}\,\mathbf{ber}(\mathfrak{Re}(e^{i\theta}X))&=
\underset{\theta\in
\mathbb{R}
}{\sup}\,\underset{\lambda\in
\Omega
}{\sup}\,\left|\left\langle \mathfrak{Re}(e^{i\theta}X)\hat{k}_{\lambda}, \hat{k}_{\lambda}\right\rangle\right|\\&=\underset{\lambda\in
\Omega
}{\sup}\,\underset{\theta\in
\mathbb{R}
}{\sup}\,\left|\left\langle \mathfrak{Re}(e^{i\theta}X)\hat{k}_{\lambda}, \hat{k}_{\lambda}\right\rangle\right|\\&=\underset{\lambda\in
\Omega
}{\sup}\left|\left\langle X\hat{k}_{\lambda}, \hat{k}_{\lambda}\right\rangle\right|\\&=\mathbf{ber}(X).
\end{align*}
The proof of the second equation is similar.
\end{proof}
\begin{remark}
If $X=H+iK$ be the certain decomposition of the operator $X$, then by using this fact
$$|\langle Hx,x\rangle|\leqslant|\langle Xx, x\rangle|=|\langle Hx,x\rangle+i\langle Kx,x\rangle|=\sqrt{|\langle Hx,x\rangle|^2+|\langle Kx,x\rangle|^2}\qquad(x\in{\mathscr H})$$ and Lemma \ref{L1}, we have
$$\mathbf{ber}(H)=\mathbf{ber}\left(\mathfrak{Re}(X)\right)\leqslant\mathbf{ber}(X)\leqslant\sqrt{\mathbf{ber}^2(H)+\mathbf{ber}^2(K)}.$$
\end{remark}
Now, by applying  Lemma \ref{L1}, we show an upper bound for
$\mathbf{ber}(AX\pm XA^*)$.
\begin{theorem}\label{main1}
Let $A, X\in {\mathbb B}({\mathscr H}(\Omega))$. Then
{\footnotesize\begin{align*}
&\mathbf{ber}^2(AX\pm XA^*)\\&\leqslant2 \|A\|^2\left(\mathbf{ber}(H^2)+\mathbf{ber}(K^2)+\sqrt{\left(\mathbf{ber}(H^2)-\mathbf{ber}(K^2)\right)^2+\mathbf{ber}^2(HK+KH)}\right),
\end{align*}}
where $X=H+iK$ is the certain decomposition of the operator $X$.
\end{theorem}
\begin{proof}
Suppose that  $\hat{k}_{\lambda}$ is  the normalized reproducing kernel of ${\mathscr H(\Omega)}$. Then
\begin{align*}
&\left|\left\langle\mathfrak{Re}(e^{i\theta}(AX+ XA^*))\hat{k}_{\lambda},\hat{k}_{\lambda}\right\rangle\right|^2
\\&=\left|\left\langle\mathfrak{Re}\left((A\mathfrak{Re}(e^{i\theta}X)+ \mathfrak{Re}(e^{i\theta}X)A^*)\right)\hat{k}_{\lambda},\hat{k}_{\lambda}\right\rangle\right|^2\\&
\qquad\qquad\qquad\qquad\qquad (\textrm {since}\, \mathfrak{Re}(T)=\mathfrak{Re}(T^*))
\\&\leqslant\left|\left\langle(A\mathfrak{Re}(e^{i\theta}X)+ \mathfrak{Re}(e^{i\theta}X)A^*)\hat{k}_{\lambda},\hat{k}_{\lambda}\right\rangle\right|^2
\\&\qquad\qquad\qquad\qquad\qquad (\textrm {since}\, |\langle\mathfrak{Re}(T)x,x\rangle|\leqslant|\langle Tx,x\rangle|)\\&
\leqslant2\left(\left|\left\langle A\mathfrak{Re}(e^{i\theta}X)\hat{k}_{\lambda},\hat{k}_{\lambda}\right\rangle\right|^2+ \left|\left\langle\mathfrak{Re}(e^{i\theta}X)A^*\hat{k}_{\lambda},\hat{k}_{\lambda}\right\rangle\right|^2\right)
 \\&\quad\quad\quad\qquad(\textrm {by the triangular inequality and the convexity\,} f(t)=t^2)
 \\&=2\left(\left|\left\langle \mathfrak{Re}(e^{i\theta}X)\hat{k}_{\lambda},A^*\hat{k}_{\lambda}\right\rangle\right|^2+ \left|\left\langle A^*\hat{k}_{\lambda},\mathfrak{Re}(e^{i\theta}X)\hat{k}_{\lambda}\right\rangle\right|^2\right)
\\&\leqslant2\left(\|A^*\|^2\left\|\mathfrak{Re}(e^{i\theta}X)\hat{k}_{\lambda}\right\|^2+\|A^*\|^2\left\|\mathfrak{Re}(e^{i\theta}X)\hat{k}_{\lambda}\right\|^2\right)
\\&=4\|A\|^2\left\langle\mathfrak{Re}(e^{i\theta}X)\hat{k}_{\lambda},\mathfrak{Re}(e^{i\theta}X)\hat{k}_{\lambda}\right\rangle
\\&=4\|A\|^2\left\langle(\mathfrak{Re}(e^{i\theta}X))^2\hat{k}_{\lambda},\hat{k}_{\lambda}\right\rangle.
\end{align*}
It follows from
\begin{align*}
(\mathfrak{Re}(e^{i\theta}X))^2&=(\mathfrak{Re}(e^{i\theta}(H+iK)))^2\\&
=(\cos\theta H-\sin\theta K)^2\\&=\cos^2\theta H^2+\sin^2\theta K^2-\cos\theta\sin\theta (HK+KH)
\end{align*}
that
{\footnotesize\begin{align*}
&\underset{\theta\in
\mathbb{R}
}{\sup}\,\left\langle\left(\mathfrak{Re}(e^{i\theta}X)\right)^2\hat{k}_{\lambda},\hat{k}_{\lambda}\right\rangle
\\&=\underset{\theta\in
\mathbb{R}
}{\sup}\Big(\cos^2\theta \left\langle H^2\hat{k}_{\lambda},\hat{k}_{\lambda}\right\rangle+\sin^2\theta \left\langle K^2\hat{k}_{\lambda},\hat{k}_{\lambda}\right\rangle-\cos\theta\sin\theta \left(\left\langle (HK+ KH)\hat{k}_{\lambda},\hat{k}_{\lambda}\right\rangle\right)\Big)\\&
\leqslant\underset{\theta\in
\mathbb{R}
}{\sup}\Big(\cos^2\theta \mathbf{ber}(H^2)+\sin^2\theta \mathbf{ber}( K^2)- \cos\theta\sin\theta \left(\left\langle (HK+ KH)\hat{k}_{\lambda},\hat{k}_{\lambda}\right\rangle\right)\Big)\\&\leqslant\frac{1}{2}\left(\mathbf{ber}(H^2)+\mathbf{ber}(K^2)
+\sqrt{\left(\mathbf{ber}(H^2)-\mathbf{ber}(K^2)\right)^2+\left(\left\langle (HK+ KH)\hat{k}_{\lambda},\hat{k}_{\lambda}\right\rangle\right)^2}\right),
\end{align*}}
whence
\begin{align*}
&\underset{\theta\in
\mathbb{R}
}{\sup}\,\left|\left\langle\mathfrak{Re}\left(e^{i\theta}(AX+ XA^*)\right)\hat{k}_{\lambda},\hat{k}_{\lambda}\right\rangle\right|^2\\&\leqslant2\|A\|^2\Big(\mathbf{ber}(H^2)+\mathbf{ber}(K^2)
\\&\quad+\sqrt{\left(\mathbf{ber}(H^2)-\mathbf{ber}(K^2)\right)^2+\left(\left\langle (HK+ KH)\hat{k}_{\lambda},\hat{k}_{\lambda}\right\rangle\right)^2}\Big).
\end{align*}
Taking the supremum over all $\lambda\in\Omega$ and using Lemma \ref{L1},  we get
{\footnotesize\begin{align}\label{waw}
&\mathbf{ber}^2(AX+ XA^*)\nonumber\\&=\underset{\theta\in
\mathbb{R}
}{\sup}\,\mathbf{ber}(\mathfrak{Re}(e^{i\theta}(AX+ XA^*)))\\&\leqslant2\|A\|^2\Big(\mathbf{ber}(H^2)+\mathbf{ber}(K^2)+\sqrt{\left(\mathbf{ber}(H^2)-\mathbf{ber}(K^2)\right)^2+\mathbf{ber}^2(HK+KH)}\Big).
\end{align}}
Replacing $A$ by $iA$ in \eqref{waw}, we have
{\footnotesize\begin{align*}
&\mathbf{ber}^2(AX- XA^*)\\&\leqslant2\|A\|^2\Big(\mathbf{ber}(H^2)+\mathbf{ber}(K^2)+\sqrt{\left(\mathbf{ber}(H^2)-\mathbf{ber}(K^2)\right)^2
+\mathbf{ber}^2(HK+KH)}\Big).
\end{align*}}
Hence
{\footnotesize\begin{align*}
&\mathbf{ber}^2(AX\pm XA^*)\\&\leqslant2\|A\|^2\Big(\mathbf{ber}(H^2)+\mathbf{ber}(K^2)+\sqrt{\left(\mathbf{ber}(H^2)-\mathbf{ber}(K^2)\right)^2+
+\mathbf{ber}^2(HK+KH)}\Big)
\end{align*}}
as required.
\end{proof}
\bigskip Theorem \ref{main1} includes a special case as follows.
\begin{corollary}
Let $A, X\in {\mathbb B}({\mathscr H}(\Omega))$
Then\\
$({\rm i})$ If $HK+KH=0$, then $\mathbf{ber}(AX\pm XA^*)\leqslant {2}\|A\|\max\left(\mathbf{ber}^{\frac{1}{2}}(H^2),\mathbf{ber}^{\frac{1}{2}}(K^2)\right).$\\
$({\rm ii})$ If $X$ is self-adjoint, then $\mathbf{ber}(AX\pm XA^*)\leqslant {2}\|A\|\mathbf{ber}^{\frac{1}{2}}(X^2).$\\
$({\rm iii})$ If $X$ is self-adjoint, then $\mathbf{ber}(AX)\leqslant \|A\|\mathbf{ber}^{\frac{1}{2}}(X^2)$.\\
where $X=H+iK$ is the certain decomposition of the operator $X$.
\end{corollary}
\begin{proof}
The first  inequality follows from Theorem \ref{main1} and the inequality
{\footnotesize\begin{align*}
&\mathbf{ber}^2(AX\pm XA^*)\\&\leqslant2\|A\|^2\Big(\mathbf{ber}(H^2)+\mathbf{ber}(K^2)+\sqrt{\left(\mathbf{ber}(H^2)-\mathbf{ber}(K^2)\right)^2}
\Big)\\&=2\|A\|^2\left(\mathbf{ber}(H^2)+\mathbf{ber}(K^2)+\left|\mathbf{ber}(H^2)-\mathbf{ber}(K^2)\right|\right)\\&
=4\|A\|^2\max\left(\mathbf{ber}(H^2),\mathbf{ber}(K^2)\right).
\end{align*}}
The second inequality follows from Theorem \ref{main1} and the hypotheses
$X=H+0i$. For the third inequality we have
\begin{align*}
\mathbf{ber}(AX)&=\underset{\theta \in
\mathbb{R}
}{\sup }\,\mathbf{ber}\left(\mathfrak{Re}(e^{i\theta}AX)\right)\\&=\frac{1}{2}\underset{\theta \in
\mathbb{R}
}{\sup }\,\mathbf{ber}\left(e^{i\theta}AX+e^{-i\theta}XA^*\right)\\&\leqslant\|A\|\mathbf{ber}^{\frac{1}{2}}(X^2)\qquad(\textrm{by part ({\rm ii})})
\end{align*}
as required.
\end{proof}
The following theorem gives some upper bounds for $\mathbf{ber}(AX\pm XA)$.
\begin{theorem}
Let $A,  X\in {\mathbb B}({\mathscr H}(\Omega))$.  Then
\begin{align*}
&({\rm i})\,\,\mathbf{ber}(AX\pm XA)\leqslant \mathbf{ber}^{\frac{1}{2}}\left(A^*A+AA^*\right)\mathbf{ber}^{\frac{1}{2}}\left(X^*X+XX^*\right).\\&
({\rm ii})\,\,\mathbf{ber}(AX\pm XA)\leqslant \mathbf{ber}^{\frac{1}{2}}\left(A^*A+X^*X\right)\mathbf{ber}^{\frac{1}{2}}\left(AA^*+XX^*\right).
\end{align*}
 \end{theorem}
\begin{proof}
Let $\hat{k}_{\lambda}$ be the normalized reproducing kernel of ${\mathscr H(\Omega)}$. Then
\begin{align*}
\left|\left\langle\left(AX\pm XA\right)\hat{k}_{\lambda},\hat{k}_{\lambda}\right\rangle\right|&
\leqslant\left|\left\langle AX\hat{k}_{\lambda},\hat{k}_{\lambda}\right\rangle\right|+\left|\left\langle XA\hat{k}_{\lambda},\hat{k}_{\lambda}\right\rangle\right|\\&
=\left|\left\langle X\hat{k}_{\lambda},A^*\hat{k}_{\lambda}\right\rangle\right|+\left|\left\langle A\hat{k}_{\lambda},X^*\hat{k}_{\lambda}\right\rangle\right|\\&
\leqslant\left\|X\hat{k}_{\lambda}\right\|\left\|A^*\hat{k}_{\lambda}\right\|+\left\|A\hat{k}_{\lambda}\right\|\left\|X^*\hat{k}_{\lambda}\right\|\\&
\leqslant\left(\left\|A\hat{k}_{\lambda}\right\|^2+\left\|A^*\hat{k}_{\lambda}\right\|^2\right)^{\frac{1}{2}}
\left(\left\|X\hat{k}_{\lambda}\right\|^2+\left\|X^*\hat{k}_{\lambda}\right\|^2\right)^{\frac{1}{2}}\\&
\qquad\qquad(\textrm {by the Cauchy-Schwartz inequality})\\&
=\left|\left\langle \left(A^*A+AA^*\right)\hat{k}_{\lambda},\hat{k}_{\lambda}\right\rangle\right|^{\frac{1}{2}}
\left|\left\langle \left(X^*X+XX^*\right)\hat{k}_{\lambda},\hat{k}_{\lambda}\right\rangle\right|^{\frac{1}{2}}\\&
\leqslant\mathbf{ber}^{\frac{1}{2}}(A^*A+AA^*)\mathbf{ber}^{\frac{1}{2}}(X^*X+XX^*).
\end{align*}
Hence
\begin{align*}
\mathbf{ber}(AX\pm XA)&=\underset{\lambda \in
\Omega
}{\sup }\left|\left\langle(AX\pm XA)\hat{k}_{\lambda},\hat{k}_{\lambda}\right\rangle\right|
\\&\leqslant\mathbf{ber}^{\frac{1}{2}}(A^*A+AA^*)\mathbf{ber}^{\frac{1}{2}}(X^*X+XX^*).
\end{align*}
Now, according to the inequality
\begin{align*}
&\left\|X\hat{k}_{\lambda}\right\|\left\|A^*\hat{k}_{\lambda}\right\|+\left\|A\hat{k}_{\lambda}\right\|\left\|X^*\hat{k}_{\lambda}\right\|
\leqslant\left(\left\|A\hat{k}_{\lambda}\right\|+\left\|X\hat{k}_{\lambda}\right\|\right)^{\frac{1}{2}}
\left(\left\|A^*\hat{k}_{\lambda}\right\|+\left\|X^*\hat{k}_{\lambda}\right\|\right)^{\frac{1}{2}}\\&
\qquad\qquad\qquad\qquad\qquad\qquad(\textrm {by the Cauchy-Schwartz inequality})
\end{align*}
and a similar argument of the proof of part ({\rm i}) we get the second inequality.
\end{proof}
For the special case $A=I$, we have the next result.
\begin{corollary}\label{main3}
Let $X\in {\mathbb B}({\mathscr H}(\Omega))$. Then
\begin{align*}
&({\rm i})\,\,\mathbf{ber}^2(X)\leqslant \mathbf{ber}\left(I+X^*X\right)\mathbf{ber}\left(I+XX^*\right).\\&
({\rm ii})\,\,\mathbf{ber}^2(X)\leqslant\frac{1}{2} \mathbf{ber}\left(X^*X+XX^*\right).
\end{align*}
\end{corollary}

\begin{remark}
Corollary \ref{main3}$({\rm ii})$ is an improvement of \eqref{prop}. To see this,  note that
\begin{align*}
\mathbf{ber}^2(X)&\leqslant\frac{1}{2}\mathbf{ber}(X^*X+XX^*)\\&\leqslant\frac{\mathbf{ber}(X^*X)+\mathbf{ber}(XX^*)}{2}\\&\leqslant
\frac{\|X^*X\|+\|XX^*\|}{2}\\&=\|X\|^2.
\end{align*}
\end{remark}
In the following theorem, we present some upper bounds of $\mathbf{ber}(A^*XB)$. To achieve this propose, we need the next lemma; see \cite{KIT}.
\begin{lemma}\label{lemma2}
If $X\in {\mathbb B}({\mathscr H})$ and $x, y\in{\mathscr H}$, then $\left|\left\langle Xx,y\right\rangle\right|^2\leqslant\left\langle|X|x,x\right\rangle\left\langle|X^*|y,y\right\rangle$, in which $|X|=\left(X^*X\right)^{1\over2}$.
\end{lemma}
\begin{theorem}\label{main2}
Let $A, B, X\in {\mathbb B}({\mathscr H}(\Omega))$.  Then
\begin{align*}
&({\rm i})\,\,\mathbf{ber}^2(A^*XB)\leqslant\|X\|^2\mathbf{ber}(A^*A)\mathbf{ber}(B^*B).\\&
({\rm ii})\,\,\mathbf{ber}(A^*XB)\leqslant\frac{1}{2}\mathbf{ber}\left(B^*|X|B+A^*|X^*|A\right).
\end{align*}
\end{theorem}
\begin{proof}
If  $\hat{k}_{\lambda}$ is the normalized reproducing kernel of ${\mathscr H(\Omega)}$, then
\begin{align*}
\left|\left\langle A^*XB\hat{k}_{\lambda},\hat{k}_{\lambda}\right\rangle\right|^2
&=\left|\left\langle XB\hat{k}_{\lambda},A\hat{k}_{\lambda}\right\rangle\right|^2
\\&\leqslant\left\|XB\hat{k}_{\lambda}\right\|^2\left\|A\hat{k}_{\lambda}\right\|^2
\\&\leqslant\|X\|^2\left\|B\hat{k}_{\lambda}\right\|^2\left\|A\hat{k}_{\lambda}\right\|^2
\\&\leqslant\|X\|^2\left\langle B\hat{k}_{\lambda},B\hat{k}_{\lambda}\right\rangle\left\langle A\hat{k}_{\lambda},A\hat{k}_{\lambda}\right\rangle
\\&=\|X\|^2\left\langle B^*B\hat{k}_{\lambda},\hat{k}_{\lambda}\right\rangle\left\langle A^*A\hat{k}_{\lambda},\hat{k}_{\lambda}\right\rangle
\\&\leqslant\|X\|^2 \mathbf{ber}(A^*A)\mathbf{ber}(B^*B),
\end{align*}
whence  $\mathbf{ber}^2(A^*XB)=\underset{\lambda \in
\Omega
}{\sup }\left|\left\langle A^*XB\hat{k}_{\lambda},\hat{k}_{\lambda}\right\rangle\right|^2\leqslant\|X\|^2\mathbf{ber}(A^*A)\mathbf{ber}(B^*B)$, and so we get the first inequality.
Also, we have
\begin{align*}
\left|\left\langle A^*XB\hat{k}_{\lambda},\hat{k}_{\lambda}\right\rangle\right|
&=\left|\left\langle XB\hat{k}_{\lambda},A\hat{k}_{\lambda}\right\rangle\right|
\\&\leqslant\left\langle |X|B\hat{k}_{\lambda},B\hat{k}_{\lambda}\right\rangle^{\frac{1}{2}}\left\langle |X^*|A\hat{k}_{\lambda},A\hat{k}_{\lambda}\right\rangle^{\frac{1}{2}}\qquad(\textrm{by Lemma \ref{lemma2}})
\\&\leqslant\frac{1}{2}\left(\left\langle \left(B^*|X|B\right)\hat{k}_{\lambda},\hat{k}_{\lambda}\right\rangle+\left\langle \left(A^*|X^*|A\right)\hat{k}_{\lambda},\hat{k}_{\lambda}\right\rangle\right)
 \\&\qquad(\textrm {by the convexity\,} f(t)=t^2)
\\&=\frac{1}{2}\left(\left\langle \left(B^*|X|B+A^*|X^*|A\right)\hat{k}_{\lambda},\hat{k}_{\lambda}\right\rangle\right)
\\&\leqslant\frac{1}{2}\mathbf{ber}\left( B^*|X|B+A^*|X^*|A\right).
\end{align*}
Hence
\begin{align*}
\mathbf{ber}(A^*XB)&=\underset{\lambda \in
\Omega
}{\sup }\left|\left\langle A^*XB\hat{k}_{\lambda},\hat{k}_{\lambda}\right\rangle\right|
\\&\leqslant\frac{1}{2}\mathbf{ber}\left( B^*|X|B+A^*|X^*|A\right).
\end{align*}
\end{proof}
In the special case of Theorem \ref{main2}, for $X=I$ we obtain the next result.
\begin{corollary}
Let $A, B, X\in {\mathbb B}({\mathscr H}(\Omega))$.  Then
\begin{align*}
&({\rm i})\,\,\mathbf{ber}^2(A^*B)\leqslant\mathbf{ber}(A^*A)\mathbf{ber}(B^*B).\\&
({\rm ii})\,\,\mathbf{ber}(A^*B)\leqslant\frac{1}{2}\mathbf{ber}(A^*A+B^*B).
\end{align*}
\end{corollary}
\begin{corollary}
Let $A, B, X\in {\mathbb B}({\mathscr H}(\Omega))$.  Then
\begin{align*}
&({\rm i})\,\,\mathbf{ber}(A^*XB)\leqslant\mathbf{ber}^{\frac{1}{2}}\left(B^*|X|B\right)\mathbf{ber}^{\frac{1}{2}}\left(A^*|X^*|A\right).\\&
({\rm ii})\,\,\mathbf{ber}(A^*XB)\leqslant\frac{1}{2}\mathbf{ber}\left(\frac{\|B\|}{\|A\|}B^*|X|B+\frac{\|A\|}{\|B\|}A^*|X^*|A\right).
\end{align*}
\end{corollary}
\begin{proof}
By  Theorem \ref{main2}$({\rm ii})$, we have
\begin{align}\label{eqution9}
\mathbf{ber}(A^*XB)&\leqslant\frac{1}{2}\mathbf{ber}\left(B^*|X|B+A^*|X^*|A\right)\nonumber\\&
\leqslant\frac{1}{2}\Big(\mathbf{ber}\left(B^*|X|B\right)+\mathbf{ber}\left(A^*|X^*|A\right)\Big).
\end{align}
Now, if we replace $A$ and $B$ by $tA$ and $\frac{1}{t} B\,\,(t>0)$ in inequality \eqref{eqution9}, respectively, then we get
\begin{align*}
\mathbf{ber}(A^*XB)\leqslant\frac{1}{2}\left(\frac{1}{t^2}\mathbf{ber}\left(B^*|X|B\right)+t^2\mathbf{ber}\left(A^*|X^*|A\right)\right).
\end{align*}
It follows from $$\min_{t>0}\left(\frac{1}{t^2}\mathbf{ber}\left(B^*|X|B\right)+t^2\mathbf{ber}\left(A^*|X^*|A\right)\right)
=2\mathbf{ber}^{\frac{1}{2}}\left(B^*|X|B\right)\mathbf{ber}^{\frac{1}{2}}\left(A^*|X^*|A\right)$$
that we get the first inequality. Moreover, if we replace $A$ and $B$ by
$\sqrt{\frac{\|A\|}{\|B\|}}A$ and $\sqrt{\frac{\|B\|}{\|A\|}}B$ Theorem \ref{main2}$({\rm ii})$, respectively,
 we reach the second inequality.
 \end{proof}
 Using Theorem \ref{main2}, we demonstrate some upper bounds for  $\mathbf{ber}(A^*XB+B^*YA)$.
\begin{theorem}\label{main4}
Let $A, B, X, Y\in {\mathbb B}({\mathscr H}(\Omega))$.  Then
\begin{align*}
&({\rm i})\,\,\mathbf{ber}(A^*XB+B^*YA)\leqslant\sqrt{2}\left\|\,|X|+|Y^\ast|\,\right\|\mathbf{ber}^{\frac{1}{2}}\left(B^*B\right)\mathbf{ber}^{\frac{1}{2}}\left(AA^*\right).\\&
({\rm ii})\,\,\mathbf{ber}(A^*XB+B^*YA)\leqslant2\sqrt{\|X\|\|Y\|}\mathbf{ber}^{\frac{1}{2}}\left(B^*B\right)\mathbf{ber}^{\frac{1}{2}}\left(AA^*\right).
\end{align*}
\end{theorem}
\begin{proof}
Applying Lemma \ref{L1} and Theorem \ref{main2}(\rm{i}), we have
\begin{align}\label{popo}
\mathbf{ber}\left(\mathfrak{Re}(e^{i\alpha}(A^*XB\pm B^*YA))\right)&=\mathbf{ber}\left(\mathfrak{Re}(A^*(e^{i\alpha} X\pm e^{-i\alpha}Y^*)B)\right)\nonumber\\&\qquad\qquad(\textrm{since}\, \mathfrak{Re}(T)=\mathfrak{Re}(T^*))\nonumber\\&
\leqslant\mathbf{ber}\left(A^*(e^{i\alpha} X\pm e^{-i\alpha}Y^*)B)\right)\nonumber\\&
\qquad\qquad(\textrm{by Lemma \ref{L1} for}\, \theta=0)\nonumber\\&
\leqslant\left\|e^{i\alpha} X\pm e^{-i\alpha}Y^*\right\|\mathbf{ber}^{\frac{1}{2}}\left(B^*B\right)\mathbf{ber}^{\frac{1}{2}}\left(AA^*\right)\nonumber
\\&\qquad\qquad(\textrm{by Theorem \ref{main2}}(\rm{i})).
\end{align}
It follows from the inequalities
\begin{align*}
\|e^{i\alpha}X\pm e^{-i\alpha}Y^\ast\|&=\left\|\left[\begin{array}{cc}
 e^{i\alpha}X\pm e^{-i\alpha}Y^\ast&0\\
 0&0
 \end{array}\right]\right\|\\&
 =\left\|\left[\begin{array}{cc}
 e^{i\alpha}&e^{-i\alpha}\\
 0&0
 \end{array}\right]\left[\begin{array}{cc}
 X&0\\
 \pm Y^\ast&0
 \end{array}\right]\right\|\\&\leqslant
 \left\|\left[\begin{array}{cc}
 e^{i\alpha}&e^{-i\alpha}\\
 0&0
 \end{array}\right]\right\|\left\|\left[\begin{array}{cc}
 X&0\\
 \pm Y^\ast&0
 \end{array}\right]\right\|
 \\&=
 \sqrt{2}\,\left\|\,\left|\left[\begin{array}{cc}
 X&0\\
 \pm Y^\ast&0
 \end{array}\right]\right|\,\right\|\\&
 =\sqrt{2}\left\|(|X|^2+|Y^\ast|^2)^{1\over2}\right\|\\&
 \leqslant \sqrt{2}\left\|\,|X|+|Y^\ast|\,\right\|\\&\qquad\qquad(\textrm{applying \cite[p. 775]{ando123} to the function} \,h(t)=t^{1\over2}),
 \end{align*}
  \eqref{popo} and Lemma \ref{L1} that
 \begin{align*}
\mathbf{ber}\left(A^*XB\pm B^*YA\right)&=\underset{\alpha \in
\mathbb{R}
}{\sup }\mathbf{ber}\left(\mathfrak{Re}(e^{i\alpha}(A^*XB\pm B^*YA))\right)\\&
\leqslant\sqrt{2}\left||\,|X|+|Y^\ast|\,\right\|\mathbf{ber}^{\frac{1}{2}}\left(B^*B\right)\mathbf{ber}^{\frac{1}{2}}\left(AA^*\right).
\end{align*}
 Thus, we get the first inequality. Moreover, Using inequality \eqref{popo} we have
 \begin{align}\label{popo2}
\mathbf{ber}\left(\mathfrak{Re}(e^{i\alpha}(A^*XB\pm B^*YA))\right)&
\leqslant\left\|e^{i\alpha} X\pm e^{-i\alpha}Y^*\right\|\mathbf{ber}^{\frac{1}{2}}\left(B^*B\right)\mathbf{ber}^{\frac{1}{2}}\left(AA^*\right)\nonumber
\\&\leqslant\left(\left\|X\right\|+\left\|Y\right\|\right)\mathbf{ber}^{\frac{1}{2}}\left(B^*B\right)\mathbf{ber}^{\frac{1}{2}}\left(AA^*\right).
\end{align}
Now,  if we replace $A$ by $\sqrt{t}A$, $B$ by $\sqrt{\frac{1}{t}}B$, $X$ by $tX$ and $Y$ by $\frac{1}{t}Y\,\,(t>0)$ in inequality \eqref{popo2}, then we get
\begin{equation}\label{bobo}
\mathbf{ber}\left(\mathfrak{Re}(e^{i\alpha}(A^*XB\pm B^*YA))\right)
\leqslant\left(t\left\|X\right\|+\frac{1}{t}\left\|Y\right\|\right)\mathbf{ber}^{\frac{1}{2}}\left(B^*B\right)\mathbf{ber}^{\frac{1}{2}}\left(AA^*\right).
\end{equation}
It follows from $\min_{t>0}\left(t\left\|X\right\|+\frac{1}{t}\left\|Y\right\|\right)=2\sqrt{\|X\|\|Y\|}$ and inequality \eqref{bobo} that
\begin{align*}
\mathbf{ber}\left(A^*XB\pm B^*YA\right)&=\underset{\alpha \in
\mathbb{R}
}{\sup }\mathbf{ber}\left(\mathfrak{Re}(e^{i\alpha}(A^*XB\pm B^*YA))\right)
\\&\leqslant2\sqrt{\|X\|\|Y\|}\mathbf{ber}^{\frac{1}{2}}\left(B^*B\right)\mathbf{ber}^{\frac{1}{2}}\left(AA^*\right).
\end{align*}
Hence, we get the second inequality.
\end{proof}
\begin{corollary}
If $A, B, X\in {\mathbb B}({\mathscr H}(\Omega))$,  then
\begin{align*}
&({\rm i})\,\,\mathbf{ber}(A^*X\pm XA)\leqslant2 \left\|X\right\|\mathbf{ber}^{\frac{1}{2}}\left(AA^*\right).\\&
({\rm ii})\,\,\mathbf{ber}(A^*B\pm B^*A)\leqslant2\mathbf{ber}^{\frac{1}{2}}\left(B^*B\right)\mathbf{ber}^{\frac{1}{2}}\left(AA^*\right).
\end{align*}
\end{corollary}
\begin{proof}
If we put $B=I$ and $X=Y$ in Theorem \ref{main4}$({\rm ii})$, then we reach the first inequality and if we take $X=Y=I$ in Theorem \ref{main4}$({\rm ii})$, then we get the second inequality.
\end{proof}
It is well known that
\begin{align}\label{hal}
 w(A^n)\leqslant w^n(A)
 \end{align}
for any $A\in {\mathbb B}({\mathscr H})$ and $n\geqslant1$. Let $\mathbb{D}=\{z\in \mathbb{C}: |z|<1\}$. Note that for any Toeplitz
operator $T_\phi$ with $\phi\in L^{\infty}(\partial \mathbb{D})$ we have $\widetilde{T}_\phi(\lambda)=\widetilde{\phi}(\lambda)\,\,(\lambda\in\mathbb{D})$, where
$\widetilde{\phi}$ is the harmonic extension of $\phi$ into $\mathbb{D}$ ( see, for instance Engli\v{s} \cite{englis}). Therefore, it is easy to see that
  \begin{align*}
 \mathbf{ber}(T_\phi)=\|\phi\|_{\infty}=\|T_\phi\|=w(T_\phi),
 \end{align*}
 which implies that $\mathbf{ber}\left((T_\phi)^n\right)\leqslant \mathbf{ber}^n(T_\phi)$ for any positive integer $n$. In general, inequality \eqref{hal} and the trivial inequality $\mathbf{ber}(A)\leqslant w(A)$ imply that
 \begin{align*}
 \mathbf{ber}(A^n)\leqslant\mathbf{ber}^n(A)\left(\frac{w(A)}{\mathbf{ber}(A)}\right)^n
 \end{align*}
 for any $A\in {\mathbb B}({\mathscr H}(\Omega))$ and $n\geqslant1$.\\
 It is natural to ask: does the same property holds  true for  the Berezin number of $A$, i.e.  is it true that $\mathbf{ber}(A^n)\leqslant\mathbf{ber}^n(A)?$\\
 Here we give some partial answers to this question.
 \begin{theorem}\label{main-karaev}
Let $A\in {\mathbb B}({\mathscr H}(\Omega))$ be an operator such that
\begin{align*}
&({\rm i})\,\,\lim_{\lambda\rightarrow \partial \mathbb{D}}\widetilde{A^n}(\lambda)\not=0\,\,\textrm{for any integer}\,\, n\geqslant1;\\&
({\rm ii})\,\,\lim_{\lambda\rightarrow \partial \mathbb{D}}\|(A^*-\widetilde{A^*}(\lambda))\hat{k}_{\lambda}\|=0.
\end{align*}
Then $\mathbf{ber}(A^n)\leqslant\mathbf{ber}^n(A)$ for any integer $n\geqslant1$.
\end{theorem}
\begin{proof}
First, let us prove by induction that if $\mathbf{ber}(A)\leqslant 1$, then $\mathbf{ber}(A^n)\leqslant1$ for any integer $n\geqslant1$.
In fact, for $n=1$ it is trivial. For $n=k$ we assume that $\mathbf{ber}(A^k)\leqslant1$, and we prove that $\mathbf{ber}(A^{k+1})\leqslant1$. We set
$L:=\mathbf{ber}(A^{k+1})$. Then $\widetilde{A^{k+1}}(\lambda)\leqslant L$ for all $\lambda\in\Omega$, and by virtue of condition $({\rm i})$,
for any sequence $(\epsilon_n)\subset(0,1)$ such that $\lim_{n\rightarrow\infty}\epsilon_n=0$ there exists the sequence
$(\lambda_n)\subset\Omega$ with $\lim_{n\rightarrow\infty}\lambda_n=\xi_0\in\partial \mathbb{D}$ such that $|\widetilde{A^{k+1}}(\lambda_n)|>L-\epsilon_n\,\,
(n\geqslant1)$. Then, by using that $\mathbf{ber}(A^k)\leqslant1$, we have
\begin{align*}
L-\epsilon_n&<|\widetilde{A^{k+1}}(\lambda_n)|\\&
=\left|\left\langle A^{k+1}\hat{k}_{\lambda_n},\hat{k}_{\lambda_n}\right\rangle\right|\\&
=\left|\left\langle A^{k}\hat{k}_{\lambda_n},A^*\hat{k}_{\lambda_n}\right\rangle\right|\\&
=\left|\left\langle A^{k}\hat{k}_{\lambda_n},A^*\hat{k}_{\lambda_n}-\widetilde{A^{*}}(\lambda_n)\hat{k}_{\lambda_n}\right\rangle+
\widetilde{A}(\lambda_n)\left\langle A^{k}\hat{k}_{\lambda_n},\hat{k}_{\lambda_n}\right\rangle\right|\\&
\leqslant\|A^k\|\|A^*\hat{k}_{\lambda_n}-\widetilde{A^{*}}(\lambda_n)\hat{k}_{\lambda_n}\|+
|\widetilde{A}(\lambda_n)||\widetilde{A^k}(\lambda_n)|\\&
\leqslant\|A^k\|\|(A^*-\widetilde{A^{*}}(\lambda_n))\hat{k}_{\lambda_n}\|+
\mathbf{ber}(A)\mathbf{ber}(A^k),
\end{align*}
whence
\begin{align*}
L\leqslant\|A^k\|\|A^*\hat{k}_{\lambda_n}-\widetilde{A^{*}}(\lambda_n)\hat{k}_{\lambda_n}\|+
1+\epsilon_n.
\end{align*}
Using condition $({\rm ii})$ we get $L\leqslant1$ whenever $n$ tends to infinity. Now, since the operator $\frac{A}{\mathbf{ber}(A)}$ also satisfies conditions
$({\rm i})$ and $({\rm ii})$, and $\mathbf{ber}\left(\frac{A}{\mathbf{ber}(A)}\right)=1$ we obtain $\mathbf{ber}\left(\left(\frac{A}{\mathbf{ber}(A)}\right)^n\right)\leqslant1$, which implies $\mathbf{ber}(A^n)\leqslant\mathbf{ber}^n(A)$, as required and this completes the proof.
\end{proof}
\begin{remark}
Every Toeplitz operator on the Hardy space ${\mathscr H}^2(\mathbb{D})$ satisfies condition $({\rm ii})$ of Theorem \ref{main-karaev}
(see Engli\v{s} \cite{englis} and Karaev \cite{kar333}), and there are many Toeplitz operators satisfying conditions of Theorem \ref{main-karaev} (see Axler and Zheng \cite{axler}, Engli\v{s}  \cite{englis} and Karaev et al. \cite{kar5}).
\end{remark}
Note that Berezin symbol has not in general multiplicative property $\widetilde{AB}=\widetilde{A}\widetilde{B}$ (for more information, see Kili\v{c} \cite{kilic}).
Our next result proves the inequality $\mathbf{ber}(AB)\leqslant \mathbf{ber}(A)\mathbf{ber}(B)$ for some operators.
\begin{proposition}\label{kara-pro}
Let $A, B\in {\mathbb B}({\mathscr H}(\Omega))$. If $\overline{\lim}_{\lambda\rightarrow\xi_0}\|(A-\widetilde{A}(\lambda))^*\hat{k}_{\lambda}\|=0$
for some $\xi_0\in\partial\mathbb{D}$, then
\begin{align*}
\overline{\lim}_{\lambda\rightarrow\xi_0}|\widetilde{AB}(\lambda)|\leqslant\mathbf{ber}(A)\mathbf{ber}(B).
\end{align*}
In particular, if  $\overline{\lim}_{\lambda\rightarrow\xi_0}|\widetilde{AB}(\lambda)|=\mathbf{ber}(AB)$, then $\mathbf{ber}(AB)\leqslant \mathbf{ber}(A)\mathbf{ber}(B)$.
  \end{proposition}
\begin{proof}
It follows from $\widetilde{A^*}=\overline{\widetilde{A}}$ that for all $\lambda\in\mathbb{D}$ we have
\begin{align*}
\widetilde{AB}(\lambda)&=\left|\left\langle AB\hat{k}_{\lambda},\hat{k}_{\lambda}\right\rangle\right|\\&
=\left|\left\langle B\hat{k}_{\lambda},A^*\hat{k}_{\lambda}\right\rangle\right|\\&
\leqslant\|B\|\left\|A^*\hat{k}_{\lambda}-\widetilde{A^*}(\lambda)\hat{k}_{\lambda}\right\|+\mathbf{ber}(A)\mathbf{ber}(B),
\end{align*}
from which by using the hypotheses  of the theorem, we have that there exists a point $\xi_0\in\partial\mathbb{D}$ such that
$\overline{\lim}_{\lambda\rightarrow\xi_0}|\widetilde{AB}(\lambda)|\leqslant\mathbf{ber}(A)\mathbf{ber}(B),$ as desired. The second assertion
of the theorem is immediate from the first one. The proposition is proved.
\end{proof}
\begin{proposition}
If $A, B\in {\mathbb B}({\mathscr H}(\Omega))$ and  $\widetilde{AB}(\lambda)\rightarrow0$ whenever $\lambda\rightarrow\partial\Omega$, then there exists
a point $\lambda_0\in\Omega$ such that
\begin{align*}
\mathbf{ber}(AB)- \mathbf{ber}(A)\mathbf{ber}(B)\leqslant\sqrt{\widetilde{B^*B}(\lambda_0)(\widetilde{AA^*}(\lambda_0)-|\widetilde{A}(\lambda_0)|^2)}.
\end{align*}
  \end{proposition}
\begin{proof}
By the same argument as in the proof of Proposition \ref{kara-pro}, we have
\begin{align*}
\widetilde{AB}(\lambda)&
\leqslant\mathbf{ber}(A)\mathbf{ber}(B)+\|B\hat{k}_{\lambda}\|\left\|A^*\hat{k}_{\lambda}-\widetilde{A^*}(\lambda)\hat{k}_{\lambda}\right\|\\&
=\mathbf{ber}(A)\mathbf{ber}(B)+\sqrt{\widetilde{B^*B}(\lambda)(\widetilde{AA^*}(\lambda)-|\widetilde{A}(\lambda)|^2)}
\end{align*}
for $\lambda\in\Omega$. Since the set $\left\{\left\langle AB\hat{k}_{\lambda},\hat{k}_{\lambda}\right\rangle:\lambda\in\Omega\right\}$
is bounded, there exists a sequence $(\lambda_n)\subset \Omega$ such that $\mathbf{ber}(AB)=\underset{\lambda \in
\Omega
}{\sup }\big|\widetilde{AB}(\lambda)\big|=\lim_{n\rightarrow\infty}\left|\left\langle AB\hat{k}_{\lambda_n},\hat{k}_{\lambda_n}\right\rangle\right|.$
On the other hand, by the hypotheses $\widetilde{AB}(\lambda)\rightarrow0$ whenever $\lambda\rightarrow\partial\Omega$, and hence the sequence $(\lambda_n)$ can not approach to the boundary $\partial\Omega$. This shows there exists $\lambda_0\in\Omega$ such that $\lim_{n\rightarrow\infty}\lambda_n=\lambda_0$. Then we obtain
from the last inequality that
\begin{align*}
\mathbf{ber}(AB)-\mathbf{ber}(A)\mathbf{ber}(B)\leqslant\sqrt{\widetilde{B^*B}(\lambda_0)(\widetilde{AA^*}(\lambda_0)-|\widetilde{A}(\lambda_0)|^2)}.
\end{align*}
\end{proof}
\textbf{Acknowledgement.} The first author would like to thank the Tusi Mathematical Research Group (TMRG). The second author was supported by King Saud
University, Deanship of Scientific Research, College of Science Research Center.

\bigskip

\end{document}